\documentclass[10pt,twocolumn]{article}

\usepackage[numbers,sort&compress]{natbib}

\newcommand{\PlotWidth}{0.96\linewidth}
\newcommand{\GapPlotWidth}{0.77\linewidth}
\newcommand{\ShortPlotHeight}{0.36\linewidth}
\newcommand{\MediumPlotHeight}{0.55\linewidth}
\newcommand{\TallPlotHeight}{0.82\linewidth}
\newcommand{\WidePlotWidth}{0.96\linewidth}
\newcommand{\WidePlotHeight}{0.54\linewidth}
\newcommand{\InputProfileFigureBegin}{\begin{figure}[t]}

\newcommand{\PaperTitle}{Sampled-data optimal control of linear fractional-order systems}
\newcommand{\PaperAuthor}{Pantelis Sopasakis}
\newcommand{\PaperEmail}{p.sopasakis@qub.ac.uk}
\newcommand{\PaperKeywords}{Sampled-data optimal control, Fractional-order systems, Bernstein polynomials}

\newcommand{\PaperAbstract}{%
This paper studies the problem of optimal control of Caputo
linear fractional-order systems of order $\alpha \in (0, 1]$
with continuous-time state constraints and input constraints.
The resulting problems are semi-infinite, since the state constraints
are imposed over the sampling intervals and not merely at the
sampling instants.
To address this, we derive a sampled-data representation of
the inter-sample trajectory and use Bernstein polynomials to
construct polytopic approximations of the continuous-time state
constraints.
Explicit bounds are obtained for the conservatism introduced by
this approximation.
By subdividing each sampling interval and applying the
Bernstein-based construction on every subinterval,
the approximation is made arbitrarily tight.
Numerical examples illustrate the tightness of the proposed
enclosures, the computational tractability of the proposed method,
and the trade-off between approximation accuracy
and computational complexity.
}

\usepackage{amsmath}
\usepackage{amssymb}
\usepackage{amsthm}
\usepackage{graphicx}
\usepackage{xcolor}
\usepackage{mathtools}
\usepackage{hhline}
\usepackage{psfrag}
\usepackage{hyphenat}
\usepackage{nicefrac}
\usepackage{framed}
\usepackage{units}

\definecolor{yellow}{rgb}{1,1,0}

\newcommand{\R}{{\rm I\!R}}
\newcommand{\E}{{\rm I\!E}}
\newcommand{\F}{\mathcal{F}}
\newcommand{\smallmat}[1]{\left[\begin{smallmatrix}#1\end{smallmatrix}\right]}

\newcommand{\conv}{\operatorname{co}}
\newcommand{\minimize}{\operatorname*{Minimize}}

\newcommand{\Var}{\operatorname{Var}}
\newcommand{\gap}{\operatorname{gap}}

\newcommand{\dfn}{\mathrel{\mathop:}=}
\newcommand{\dd}{\mathrm{d}}

\usepackage{tikz}
\usetikzlibrary{external}
\usepackage{pgfplots}
\usepgfplotslibrary{fillbetween}
\pgfplotsset{compat=1.18}

\providecommand{\PlotWidth}{0.55\linewidth}
\providecommand{\GapPlotWidth}{0.55\linewidth}
\providecommand{\ShortPlotHeight}{0.20\linewidth}
\providecommand{\MediumPlotHeight}{0.30\linewidth}
\providecommand{\TallPlotHeight}{0.45\linewidth}
\providecommand{\WidePlotWidth}{0.65\linewidth}
\providecommand{\WidePlotHeight}{0.35\linewidth}
\providecommand{\InputProfileFigureBegin}{\begin{figure}}

\graphicspath{{figures/}}

\theoremstyle{plain}
\newtheorem{theorem}{Theorem}
\newtheorem{lemma}{Lemma}
\newtheorem{proposition}{Proposition}
\theoremstyle{remark}
\newtheorem{remark}{Remark}

\usepackage{microtype}
\usepackage[hidelinks]{hyperref}
\hypersetup{
  pdftitle={\PaperTitle},
  pdfauthor={\PaperAuthor},
  pdfkeywords={\PaperKeywords}
}

\title{\PaperTitle}
\author{%
  \PaperAuthor\\
  \small School of Electronics, Electrical Engineering, and Computer Science (EEECS)\\
  \small Queen's University Belfast, Belfast BT9 5AG, United Kingdom\\
  \small \texttt{\PaperEmail}%
}
\date{}

\begin{document}
\maketitle

\begin{abstract}
\PaperAbstract
\end{abstract}

{\raggedright\noindent\textbf{Keywords:} \PaperKeywords\par}
\medskip

\section{Introduction}

\subsection{Motivation \& background}
% WHY FRACTIONAL-ORDER SYSTEMS ARE WORTH STUDYING
% (APPLICATIONS / MOTIVATION)
Fractional-order systems are dynamical systems involving 
derivatives of non-integer order, which, unlike regular 
derivatives, are non-local operators \citep{Monje2010}.
Such models are particularly well suited for 
describing memory 
effects that are often difficult to capture 
with classical integer-order dynamics.
Fractional-order derivatives arise also naturally 
from the study of anomalous diffusion \citep{Zacher2019}. 
Models with fractional-order derivatives have found 
applications in a wide range of domains, including 
pharmacokinetics \citep{Sopasakis+2017}, 
energy storage systems \citep{ZouLei+2018}, 
epidemiology \citep{Khondoker+2021,Ding+2012}, 
and the modelling of nerve cells \citep{Langlands2009}.

% RECENT RESEARCH RESULTS IN FRACTIONAL OPTIMAL CONTROL
The theory of optimal control for fractional-order 
dynamical systems has witnessed significant progress 
in recent years.
In the linear-quadratic setting, the role played by the Riccati 
differential equation in the classical case is taken by a 
Fredholm-type integral equation \citep{Gomoyunov2025}.
A Pontryagin-type maximum principle has been developed 
for optimal control problems with fractional-order 
dynamics and state and input constraints \citep{Moon2025}.
Alongside, numerical methods have been developed for 
solving such problems with linear \citep{Malmir2024} and 
nonlinear dynamics \citep{LiuGong+2021}.
Discretization methods based on the Gr{\"u}nwald-Letnikov 
derivative have been proposed to facilitate the solution 
of optimal control problems, however these disregard the 
discretization error \citep{YaghiniBagheri+2024,SopasakisSarimveis2017}.

% WHY DO WE NEED SAMPLED-DATA OPTIMAL CONTROL?
% And what are the challenges?
In many applications, control actions are generated 
digitally and applied through a zero-order hold element.
As a result, imposing state constraints over the entire 
prediction horizon leads to a semi-infinite optimal control 
problem, since the constraints must be enforced at all 
inter-sample times and not only at the sampling instants. This 
makes the problem substantially more challenging than its 
discrete-time counterpart.

% SAMPLED-DATA METHODS HAVE BEEN STUDIES FOR INTEGER-ORDER SYSTEMS
Sampled-data optimal control for integer-order constrained 
systems has received considerable attention.
In particular, conditions for recursive feasibility and asymptotic 
stability are derived in \citep{EsterhuizenWorthmann+2021} 
and \citep{SopPatSar13},
enabling the design of model predictive controllers
with guaranteed satisfaction of the state constraints in 
continuous time.
For constrained input-affine sampled-data systems with bounded 
additive disturbances, Cortez et al.~\cite{CortezDrgona+2022} propose a
custom sampled-data control barrier function. 
Related ideas have also been developed for constrained nonlinear 
systems, where Ding et al.~\cite{DingBhaveDeka2026} combine control barrier
functions with backstepping to compute reach-avoid sets and design 
model predictive controllers with guaranteed recursive feasibility.

% WHAT IS THE GAP IN THE LITERATURE?
% Why can't we apply existing methods to fractional-order systems?
Sampled-data approaches for integer-order systems, however, 
are not directly applicable to fractional-order systems. 
A key reason is that they rely on the Markov property where future 
states are a function only of the current state and input. 
Fractional-order systems, by contrast, are inherently nonlocal in time, 
so their evolution depends on the past trajectory rather than 
on the current state \citep{Monje2010}.
Another challenge is that the trajectories of fractional-order 
systems are described by Mittag-Leffler functions with irregular
behavior at the sampling time instants.
Sampled-data fractional-order optimal control formulations 
constitute a largely unexplored area. One of the few exceptions is 
the work of \cite{LiWenLiu2021}, which focuses on 
fractional-order integrators and does not consider state 
constraints.

% PROBLEM STATEMENT OF THIS PAPER
% What is the research question? (high level)
In this paper, we consider optimal control of linear 
fractional-order systems with the Caputo fractional derivative
under zero-order-hold actuation 
and continuous-time state and input constraints. 
This leads to a semi-infinite optimization problem, 
since the state constraints must hold over continuous-time
sampling intervals. 
Our approach is to construct a tractable finite-dimensional 
approximation of the state constraints and
quantify the resulting conservatism, 
and study the trade-off between 
approximation tightness and computational complexity.

% OUR APPROACH
Our approach uses Bernstein polynomials 
\citep{Phillips2003} and a Lipschitz-type property 
to construct safe polytopic overapproximations 
of the inter-sample trajectories of sampled-data 
fractional-order systems.
This is somewhat similar in spirit to the approach of 
Allamaa et al.~\cite{AllamaaPatrinos+2023} where Bernstein polynomials
are used to compute a safety envelope to enforce 
continuous-time state/input constraints. 
Furthermore, we quantify the approximation error and show 
that by subdividing the inter-sample intervals and applying 
the Bernstein-based overapproximation to each subinterval, 
we can obtain an arbitrarily tight approximation.
We illustrate numerically the resulting trade-off between 
tightness and computational complexity and demonstrate
that the resulting optimal 
control problems come with a large number of convex 
constraints, but can be solved very efficiently
using modern QP solvers such as OSQP~\citep{osqp2020}.

\subsection{Notation and preliminaries}
For $\epsilon>0$ we denote by 
$\mathcal{B}_\epsilon=\{x\in\R^n:\|x\|\leq
\epsilon\}$ the closed ball of $\R^n$ with radius $\epsilon$.  
We use the notation $a \wedge b = \min\{a, b\}$,
$(a)_+ = \max\{0, a\}$.
For a normed space $X$, we say that a function $f:\R\to X$
is $L$-smooth if $f$ is differentiable and $f'$ is $L$-Lipschitz, 
i.e., $\|f'(x)-f'(y)\| \leq L |x-y|$ for $x, y\in \R$.
The support function of a nonempty set $C \subseteq \R^n$ 
is defined as $\delta^*_C(y) = \sup_{x\in C}\langle y, x\rangle$.
The inner, outer, and Painlev\'e-Kuratowski limits of sequences 
of sets are defined as in \citep[Chap.~4]{RocWet98}.
The Minkowski sum of sets is denoted by $\oplus$.
The convex hull of a set $A\subseteq X$ is denoted by 
$\conv A$; this is the smallest convex set that contains $A$.
Lastly, the two-parameter Mittag-Leffler function is 
\begin{equation}
    E_{\alpha,  \beta}(z)
    {}={}
    \sum_{j=0}^{\infty}
    \frac{z^j}{\Gamma(\alpha j + \beta)},
\end{equation}
and the (one-parameter) Mittag-Leffler function is 
$E_{\alpha}(z) = E_{\alpha, 1}(z)$.
Note that this function generalizes the exponential; 
indeed, $E_1(z) = \exp(z)$.

\section{Fractional-order Systems}
In this section we recall the solution formula for linear fractional-order 
systems and specialize it to sampled-data systems with a zero-order hold. 
The Caputo derivative of order  $\alpha\in (0,1)$
of an almost everywhere differentiable function $x:[0, 1]\to\R^n$
is defined as \citep{Monje2010}
\begin{equation}
    \mathrm{D}^\alpha x(t)
    {}={}
    \frac{1}{\Gamma(1-\alpha)}
    \int_0^t (t - s)^{-\alpha}\dot{x}(s)\dd s,
\end{equation}
where $\Gamma$ is the Euler gamma function, 
provided the integral exists. For $\alpha=1$, 
$\mathrm{D}^\alpha x(t) \coloneqq \frac{\dd}{\dd t}x(t)$.

Consider the following linear fractional system%
\begin{align}\label{eq:flti}
    \mathrm{D}^\alpha x(t) = Ax(t) + Bu(t),
\end{align}
where $x\in\R^n$ is the system state, $u\in\R^m$ is 
the input variable and $A$ and $B$ are matrices of appropriate
dimensions. 
Note that calling $x$ a state is in fact a misnomer, because this 
system does not have the Markov property.
Given the initial condition $x(0)=x_0$, the solution
of~\eqref{eq:flti} is~\citep{Kac10}
\begin{align}
    x(t) = \mu(t)x_0 + \int_0^t \Phi(t-\tau)Bu(\tau)\dd\tau,
    \label{eq:system_response_ct}
\end{align}
where $\mu(t) = E_\alpha(At^\alpha)$  and
$\Phi:(0, \infty)\to \R^{n\times n}$ is the function
\begin{align}
    \Phi(t)=\sum_{j=0}^{\infty}\frac{A^jt^{(j+1)\alpha -1}}{\Gamma((j+1)\alpha)}.
    \label{eq:Phi}
\end{align}

Let $(u_k)_{k}$ be a sequence of control 
actions which are applied to the continuous-time system through
a zero-order hold element, i.e., $u(t)=u_k$ for 
$t\in [kh, (k+1)h)$, where $h$ is the sampling period. 
Then, the system response of \eqref{eq:system_response_ct} becomes
\begin{equation}
    x(t) 
    {}={} 
    \mu(t)x_0 
    {}+{} 
    \sum_{i=0}^{\lfloor t/h \rfloor}\sigma_i(t) B u_i
    \label{eq:solution:1}
\end{equation}
and the functions $\sigma_i:\R_+ \to \R^{n\times n}$ for 
$i=0,\ldots, \lfloor t/h \rfloor$ are defined as 
\begin{equation}
    \sigma_i(t) = \int_{ih}^{t \wedge (i+1)h}\Phi(t-\tau)\dd\tau,
    \label{eq:sigma_i:def}
\end{equation}
and $\sigma_i(t) = 0$ for $i > \lfloor t/h \rfloor$.
The $\sigma_i$ can be written in a compact and convenient way as we show next.

% -----------------------------------
% PROPOSITION 1
% REPRESENTATION OF sigma_i(t)
% -----------------------------------
% \begin{framed}
\begin{proposition}\label{prop:sigma_i_G}
    The functions $\sigma_i$ are given by 
    % There is also this representation, which is not so convenient...
    % \begin{equation}
    %     \sigma_i(t) 
    %     {}={} 
    %     \sum_{j=0}^{\infty}
    %         \frac{A^j  [(t-ih)^{(j+1)\alpha} - (t - s_i(t))^{(j+1)\alpha}]}{\Gamma((j+1)\alpha + 1)},
    % \end{equation}
    \begin{equation}
        \sigma_i(t) 
        {}={}
        G((t-ih)_+^\alpha)
        - 
        G((t-(i+1)h)_+^\alpha),
    \end{equation}
    for $t \geq 0$, where $G(\rho) = \rho E_{\alpha,  \alpha+1}(A\rho)$,
    for $\rho \geq 0$.
\end{proposition}
% \end{framed}
\begin{proof}
Let us define $F(r) = G(r^\alpha)$ for $r\geq 0$.
By the definition of $G$, for $\rho\geq 0$ we have 
\begin{equation}
    F(r)
    {}={}
    \sum_{j=0}^{\infty}
    \underbracket[0.5pt]{\frac{A^j r^{(j+1)\alpha}}{\Gamma((j+1)\alpha+1)}}_{F_j(r)}.
    \label{eq:F_series:def}
\end{equation}

We will show that $F'(r) = \Phi(r)$.
Consider an interval $[a, b]\subset(0,\infty)$.
Following \citep[Cor. 3.7.3]{TerenceTaoAnalysisII},
note that $F_j$ are continuously differentiable and 
the series in \eqref{eq:F_series:def} 
converges; we need to show that 
$\sum_{j=0}^{\infty} \sup_{a\leq r\leq b}\|F_j'(r)\| < \infty$.
Indeed, for $r\in[a,b]$,
\begin{align}    
    \| F_j'(r)\|
    {}={}&
    \left\|
        \frac{A^j (j+1)\alpha r^{(j+1)\alpha-1}}{\Gamma((j+1)\alpha+1)}
    \right\|
    \notag\\
    {}={}&
    \left\|
        \frac{A^j r^{(j+1)\alpha-1}}{\Gamma((j+1)\alpha)}
    \right\|
    {}\leq{}
    \frac{\|A\|^j b^{(j+1)\alpha}}{a \Gamma((j+1)\alpha)},
\end{align}
where we used the fact that $r^{(j+1)\alpha-1} = r^{(j+1)\alpha}/r \leq b^{(j+1)\alpha}/a$.
Summability follows from the Weierstrass M-test.
As a result,
\begin{align}
    F'(r)
    {}={}&
    \sum_{j=0}^{\infty}
    F_j'(r)
    {}={}
    \sum_{j=0}^{\infty}    
    \frac{A^j r^{(j+1)\alpha-1}}
    {\Gamma((j+1)\alpha)}
    =
    \Phi(r).
    \label{eq:F_antiderivative_Phi}
\end{align}
Since $[a,b]\subset(0,\infty)$ is arbitrary, this equality holds 
for all $r>0$.
By the fundamental theorem of calculus 
$\int_a^b \Phi(r) \dd r = F(b) - F(a)$ for $0 < a \leq b$ 
and by the continuity of $G$, we have that $F$ is continuous on 
$[0, \infty)$, so this holds for $0 \leq a \leq b$.
From \eqref{eq:F_antiderivative_Phi}, 
$\tfrac{\dd}{\dd\tau}F(t-\tau) = -\Phi(t - \tau)$,
so 
\begin{equation}
    \sigma_i(t) = -F(t-\tau)\big|_{\tau=ih}^{\tau=t\wedge (i+1)h}.
\end{equation}
Using $t - t \wedge (i+1)h = (t - (i+1)h)_+$, 
\begin{equation}
    \sigma_i(t)
    =
    F((t-ih)_+)-F((t-(i+1)h)_+).
\end{equation}
Here the plus operator in the first term, is used 
to establish $\sigma_i(t) = 0$ for 
$i > \lfloor t/h \rfloor$.
Since $F(r)=G(r^\alpha)$ for $r\geq 0$, the assertion is 
proven.
\end{proof}

Our goal is to solve the following optimal control problem
with horizon $N$
\begin{subequations}
    \begin{align}
        \minimize_{u_0, \ldots, u_{N-1}}
        \;{}&{}
        J_N(x_0, u_0, \ldots, u_{N-1})
        \\
        \text{subj. to:}\;
        {}&{}
        \mathrm{D}^\alpha x(t) {=} Ax(t) + Bu(t), t{\in} [0, Nh]
        \\
        {}&{}
        x(t) \in \mathcal{X}, t\in [0, Nh]
        \label{eq:ocp:state_constraints}
        \\
        {}&{}
        u_k \in \mathcal{U}, k=0, \ldots, N-1,
        \label{eq:ocp:input_constraints}
    \end{align}
\end{subequations}
where $\mathcal{U}\subseteq\R^m$ is a nonempty 
closed convex set and 
$\mathcal{X} \dfn \{x\in \R^n: Hx \leq b\}$,
with $H\in\R^{n_c\times n}$ and $b\in\R^{n_c}$,
is a nonempty polyhedral set.
Moreover, $J_N$ is a continuous cost function.
This is a challenging problem because the state 
constraints in \eqref{eq:ocp:state_constraints}, 
albeit convex, are semi-infinite, that is, we have an 
infinite number of constraints (indexed by $t$).
Next, we use Bernstein polynomials to construct tight 
polytopic overapproximations of the state constraints.

\section{Constraint handling}
\subsection{Bernstein enclosures}
Given a function $f:[0, 1]\to X$, where $X$ is a normed space 
(e.g., $\R^n$ or $\R^{m\times n}$) we want to determine a polytope 
that overapproximates $f([0, 1])$. 
By taking points $(t_i)_i \subseteq [0, 1]$ the set 
$\conv \{f(t_i)\}_i$ is naturally an inner approximation 
in the sense $\conv \{f(t_i)\}_i \subseteq \conv f([0, 1])$.
Using Bernstein polynomials \citep{Phillips2003} we will show that under certain 
regularity conditions on $f$ we can produce an outer approximation 
by appropriately inflating $\conv \{f(t_i)\}_i$. 

The Bernstein polynomial of degree $q$ of $f$ is given by 
\begin{equation}
    B^qf(x) = \sum_{\nu=0}^{q} b_{\nu, q}(x) f(\nicefrac{\nu}{q}),
\end{equation}
where $b_{\nu, q}$ are the Bernstein basis polynomials of 
degree $q$ given by 
\begin{equation}
    b_{\nu, q}(x) 
    {}={} 
    {q \choose \nu}
    x^\nu
    (1-x)^{q-\nu},
\end{equation}
for $\nu=0,\ldots, q$; these satisfy 
$\sum_{\nu=0}^{q} b_{\nu, q}(x) = 1$ and 
$b_{\nu, q}\geq 0$, so $B^qf(x)$ is a convex combination 
of the values $f(\nu/q)$, for $\nu=0,\ldots, q$.
We will now introduce polytopic overapproximations
of $L$-smooth function images of the form $f([0, 1])$.
We call these approximations \emph{Bernstein enclosures}.

% -----------------------------------
% LEMMA 2
% Bernstein enclosures
% -----------------------------------
% \begin{framed}
\begin{lemma}[Bernstein enclosures]\label{lem:bernstein_bound}
    Let $X$ be a normed space.
    Suppose $f:[0, 1]\to X$ is $L$-smooth.
    Then,
    \begin{equation}
    f([0,1]) \subseteq \conv\{f(\nu/q)\}_{\nu=0}^{q} \oplus \mathcal{B}_{\eta},
    \label{eq:bernstein_enclosure}
    \end{equation}
    where $\eta =  \nicefrac{L}{8q}$.    
    If $f\in C^2([0, 1])$, then one may take 
    $L = \sup_{t\in [0,1]}\|f''(t)\|$.
\end{lemma}
% \end{framed}

\begin{proof}
    Suppose $f$ is $L$-smooth.
    For $x, y \in [0, 1]$, 
    \begin{equation}
        f(y) = f(x) + f'(x)(y-x) + R_x(y),
        \label{eq:pf:f_linearisation_residual}
    \end{equation}
    and the residual, $R_x$, satisfies
    $
        \|R_x(y)\| \leq \tfrac{L}{2}(y-x)^2,
    $
    and substituting $y = \nu/q$ in we have
    \begin{equation}
        \left\|R_x\left(\tfrac{\nu}{q}\right)\right\| 
        {}\leq{}
        \tfrac{L}{2}\left(\tfrac{\nu}{q}-x\right)^2.
        \label{eq:pf:resid_bound:kq}
    \end{equation}
    We substitute $f$ from \eqref{eq:pf:f_linearisation_residual} into the 
    Bernstein polynomial expansion 
    \begin{align*}
        B_q f(x)
        {}={}&
        \sum_{\nu=0}^{q}b_{\nu, q}(x)
        \bigl[f(x)+f'(x)(\nicefrac{\nu}{q}-x)
        \\[-2pt]
        &\hspace{7em}{}+ R_x(\nicefrac{\nu}{q})\bigr].
    \end{align*}
    Note that $b_{\nu,q}$ satisfy 
    $\sum_{\nu=0}^{q}b_{\nu,q}(x) \nicefrac{\nu}{q} = x$
    \citep[Sec. 7.1]{Phillips2003},
    so $B_q f(x)$ becomes
    \begin{equation}
        B_q f(x)
        {=}
        f(x)+\sum_{\nu=0}^{q}b_{\nu, q}(x)R_x(\nicefrac{\nu}{q}).
    \end{equation}
    As a result,
    \begin{align}
        \|B_qf(x) - f(x)\| 
        {}\leq{}&
        \sum_{\nu=0}^{q} b_{\nu, q}(x)\;\|R_x(\nicefrac{\nu}{q})\|
        \notag\\
        {}\overset{\eqref{eq:pf:resid_bound:kq}}{\leq}{}&
        \frac{L}{2}\sum_{\nu=0}^{q}b_{\nu, q}(x)\left(\nicefrac{\nu}{q} - x\right)^2.
        \label{eq:pf:bernstein_bound:1}
    \end{align}
    Interpreting $b_{\nu,q}(x)$ as the probabilities 
    of a binomial random variable $K\sim\mathrm{Binom}(q,x)$, we have 
    \begin{align}
        \sum_{\nu=0}^{q}b_{\nu, q}(x)\left(\nicefrac{\nu}{q} - x\right)^2
        {}={}&
        \E\left[\left(\tfrac{K}{q}-x\right)^2\right]
        \notag\\
        {}={}&
        \tfrac{1}{q^2}\Var[K]
        {}={}
        \frac{x(1-x)}{q}.
    \end{align}
    From \eqref{eq:pf:bernstein_bound:1},
    $\|B_q f(x) - f(x)\|  \leq \tfrac{Lx(1-x)}{2q}$.
    Since 
    $0 \leq x(1-x) \leq \tfrac{1}{4}$ for $0\leq x\leq 1$, 
    the result of \eqref{eq:bernstein_enclosure} follows.

    Lastly, if $f\in C^2([0, 1])$, then for $x, y\in [0, 1]$,
    \begin{equation}
        f'(y)-f'(x)=\int_x^y f''(s)\dd s,
    \end{equation}
    so
    $\|f'(y)-f'(x)\|
        \leq
        \sup_{s\in[0,1]}\|f''(s)\|\,|y-x|,$
    so one may take
    $L=\sup_{s\in[0,1]}\|f''(s)\|$.
\end{proof}

\subsection{Enclosures of state trajectories}%
\label{sec:polytopic_overapproximations}
In this section we use Proposition \ref{prop:sigma_i_G} and 
Lemma \ref{lem:bernstein_bound} to determine polytopic 
overapproximations of the state trajectories.
We will also address the challenge that $\sigma$ is not $L$-smooth.

% 1. Problem statement & objective
The solution in \eqref{eq:solution:1} for $t\in I_l \dfn [lh, (l+1)h]$
can be written as 
$x(t) = \Phi_l(t) z_l$, where $z_l=(x_0, u_0, \ldots, u_l)$, and
$\Phi_l(t) = [\mu(t) ~ \sigma_0(t)B ~ \ldots ~ \sigma_l(t)B]$. 
The system is subject to polyhedral state constraints, 
$x(t) \in \mathcal{X}$, over each interval $t\in I_l$.
The state constraints are convex but semi-infinite, so we 
aim to determine a polytope $\mathcal{Q}_l$ such that 
$H\Phi_l(t) \in \mathcal{Q}_l$ for all $t\in I_l$.
However, Lemma \ref{lem:bernstein_bound} cannot be applied 
directly because $\Phi_l(t)$ is not $L$-smooth for $0 < \alpha < 1$.
Indeed, for $t\in I_l$ it is $\sigma_l(t) = G((t-lh)^\alpha)$ and from the proof 
of Proposition \ref{prop:sigma_i_G}, 
$\frac{\dd}{\dd r}G(r^\alpha) = \Phi(r)= \frac{1}{\Gamma(\alpha)}r^{\alpha - 1} I + O(r^{2\alpha - 1})$---cf. 
Equation \eqref{eq:Phi}---therefore, $\|\Phi(r)\|\to\infty$
as $r\to 0^+$, that is, $G(r^\alpha)$ is not an $L$-smooth function of $r$. 

To overcome this issue we exploit the fact that $G$ is 
analytic. We introduce the local time variable $\tau$ with $t = h(l + \tau)$,
$\tau\in[0, 1]$ and 
observe that $\sigma_l(t) = G(h^\alpha\tau^\alpha)$ and 
$\sigma_{l-1}(t) = G(h^\alpha(1 + \tau)^\alpha) - G(h^\alpha\tau^\alpha)$.
The solution in \eqref{eq:solution:1} for $t\in I_l$ and $l\geq 1$
is then written as 
\begin{align}
    x(t)
    {}={}&
    \mu(t)x_0
    + \sum_{i=0}^{l-2}\sigma_i(t)Bu_i
    \notag\\
    &{}+ G(h^\alpha(1 + \tau)^\alpha)Bu_{l-1}
    \notag\\
    &{}+ G(h^\alpha\tau^\alpha)B\Delta u_l.
\end{align}
where $\Delta u_l = u_l - u_{l-1}$.

Note again that the function $\tau \mapsto G(h^\alpha\tau^\alpha)B$ is not $L$-smooth as
its derivative is $G'(h^\alpha\tau^\alpha)h^\alpha \alpha\tau^{\alpha-1}$.
Our approach is to isolate this function and 
treat it separately.
More specifically, we can define the function $S_l:[0,1]\to \R^{n\times d_l}$,
where $d_l = n + ml$ as 
\begin{multline}
    S_l(\tau)
    {=}
    \big[
    \mu(t) ~~ \sigma_0(t)B
    ~~ \sigma_1(t)B
    \\
    \cdots \sigma_{l-1}(t)B    
    ~~ G(h^\alpha(1 + \tau)^\alpha)B
    \big],
\end{multline}
so that 
\begin{equation}
    x(t) = S_l(\tau) z_{l-1} +   G(h^\alpha\tau^\alpha)B \Delta u_l,
    \label{eq:key-system-representation}
\end{equation}
for $t\in I_l$, where $z_{l-1}=(x_0, u_0, \ldots, u_{l-1}) \in \R^{d_l}$.
By writing the system response this way we have separated 
the singularity-free part, $S_l(\tau)$, and the 
nonsmooth part involving $G(h^\alpha \tau^\alpha)$.

For each constraint $H_i x \leq b_i$, where $H_i$ is the 
$i$-th row of $H$, we define the function 
$\psi^{\rm sm}_{l, i}:[0, 1]\to\R^{d_l}$
given by
$\psi^{\rm sm}_{l, i}(\tau) = (H_i S_l(\tau))^\intercal$, 
$\tau\in[0,1]$.
We can apply Lemma \ref{lem:bernstein_bound} to the 
$L_{l, i}^{\rm sm}$-smooth functions $\psi^{\rm sm}_{l, i}$.
We have 
\begin{equation}
    \psi^{\rm sm}_{l, i}([0, 1]) \subseteq 
    \underbracket[0.5pt]{
        \conv\{s_{l, i}^{\nu}\}_{\nu=0}^{q} \oplus \mathcal{B}_{\eta^{\rm sm}_{l, i}}
    }_{\mathcal{Q}_{l, i}^{\rm sm} \subseteq \R^{d_l}},
    \label{eq:overapproximation:sm}
\end{equation}
where $\eta^{\rm sm}_{l, i} = \nicefrac{L_{l, i}^{\rm sm}}{8q}$ and 
$s_{l, i}^{\nu} = \psi^{\rm sm}_{l, i}(\nu/q)$.

For the nonsmooth part, $\psi_{l, i}^{\rm ns}(\tau) = (H_i G(h^\alpha \tau^\alpha)B)^\intercal$, we define 
$\hat{\psi}_{l, i}^{\rm ns}(\rho) = (H_i G(h^\alpha \rho)B)^{\intercal}$, 
where $\rho = \tau^\alpha \in [0, 1]$ is a local fractional time. 
Function $\hat{\psi}_{l, i}^{\rm ns}:[0, 1]\to\R^m$ satisfies the 
conditions of Lemma \ref{lem:bernstein_bound} and 
\begin{equation}
    \psi_{l, i}^{\rm ns}([0, 1])
    =
    \hat{\psi}_{l, i}^{\rm ns}([0, 1])
    \subseteq 
    \underbracket[0.5pt]{
        \conv\{n_{l, i}^{\nu}\}_{\nu=0}^{q} \oplus \mathcal{B}_{\eta^{\rm ns}_{l, i}}
    }_{\mathcal{Q}_{l,i}^{\rm ns} \subseteq \R^m},
    \label{eq:overapproximation:ns}
\end{equation}
where $L^{\rm ns}_{l, i}$ is a Lipschitz constant of 
$\frac{\dd}{\dd\rho}\hat{\psi}^{\rm ns}_{l, i}$,
$\eta^{\rm ns}_{l, i} = \nicefrac{L^{\rm ns}_{l, i}}{8q}$,
and $n_{l, i}^{\nu} = \hat{\psi}^{\rm ns}_{l, i}(\nu/q)$.

Note that the decoupling into a smooth and a nonsmooth part introduces
some conservatism, which we address in Section 
\ref{sec:tighter_approximations}.

The case $l=0$ is handled separately since no $\Delta u$ term 
is defined. For $l=0$ it is $x(t) = S_0(\tau) z_0$
for $t\in I_0$ and $t=h\tau$, where $z_0 = (x_0, u_0)$ and 
$S_0(\tau) = [\mu(h\tau) ~ \sigma_0(h\tau)B]
= [E_\alpha(Ah^\alpha \tau^\alpha) ~ G(h^\alpha \tau^\alpha)B]$.
We define the function 
$\hat{\psi}_{0, i}(\rho) = [H_iE_\alpha(Ah^\alpha \rho) ~ H_iG(h^\alpha \rho)B]^\intercal$,
in terms of the local fractional time $\rho = \tau^\alpha$.
The function $\hat{\psi}_{0, i}$ satisfies the requirements of 
Lemma \ref{lem:bernstein_bound}, which allows to obtain a 
polytope $\mathcal{Q}_{0,i}$ such that 
$\hat\psi_{0,i}([0,1])\subseteq \mathcal{Q}_{0,i}$.

In the following section we will see how the approximations 
in \eqref{eq:overapproximation:sm} and \eqref{eq:overapproximation:ns}
can be cast as simple convex constraints within an optimal 
control problem. 

\subsection{Tractable state constraints}
The constraints $H_i x(t) \leq b_i$ for $t\in I_l$ are equivalent to 
\begin{equation}
    \psi_{l, i}^{{\rm sm}}(\tau)^\intercal z_{l-1} 
    {}+{}
    \psi_{l, i}^{\rm ns}(\tau)^\intercal \Delta u_l \leq b_i,
\end{equation}
for $\tau \in [0, 1]$. Using the polytopic overapproximations 
of \eqref{eq:overapproximation:sm} and 
\eqref{eq:overapproximation:ns} we derive the following 
sufficient condition for state constraint satisfaction
\begin{multline}
    q_{l,i}^{\rm sm\; \intercal}z_{l-1} + q_{l,i}^{\rm ns\; \intercal}\Delta u_l \leq b_i,     
    \\
    \forall q_{l,i}^{\rm sm}\in\mathcal{Q}_{l,i}^{\rm sm}, 
    \forall q_{l,i}^{\rm ns}\in\mathcal{Q}_{l,i}^{\rm ns}.
\end{multline}
This is equivalent to
\begin{align}
    &\max_{q_{l,i}^{\rm sm}\in\mathcal{Q}_{l,i}^{\rm sm}}
    q_{l,i}^{\rm sm\; \intercal}z_{l-1} 
    {}+{} 
    \max_{q_{l,i}^{\rm ns}\in\mathcal{Q}_{l,i}^{\rm ns}}
    q_{l,i}^{\rm ns\; \intercal}\Delta u_l 
    {}\leq{} b_i
    \\
    \Leftrightarrow{}
    & \delta^*_{\mathcal{Q}_{l,i}^{\rm sm}}(z_{l-1})
    + \delta^*_{\mathcal{Q}_{l,i}^{\rm ns}}(\Delta u_l )
    {}\leq{} b_i
\end{align}
Using the properties of support functions 
(i) $\delta^*_{\conv \{s_\nu\}_\nu}(z) = \max_{\nu} s_\nu^{\intercal} z$,
(ii) if $\mathcal{B}_R$ is then $\|{}\cdot{}\|$-ball of radius $R$,
 the $\delta^*_{\mathcal{B}_R}(z) = R\|z\|_{*}$, where $\|{}\cdot{}\|_*$ 
 is the dual norm,
and (iii) $\delta^*_{A\oplus B} = \delta^*_A + \delta^*_B$, we have 
\begin{align}
    &\max_{\nu=0,\ldots, q} s_{l, i}^{\nu\; \intercal} z_{l-1}
    + \eta^{\rm sm}_{l,i} \|z_{l-1}\|_*
    \notag\\
    &\quad{}+ \max_{\nu=0,\ldots, q} n_{l, i}^{\nu\; \intercal} \Delta u_{l}
    + \eta^{\rm ns}_{l,i}  \|\Delta u_{l}\|_*  \leq b_i,
    \label{eq:state_constraints_condition_max}
\end{align}
where $\|{}\cdot{}\|_*$ is the dual of the norm used in Lemma 
\ref{lem:bernstein_bound}. We can introduce the slack variables 
$\lambda_{l,i}^{\rm sm}$ and $\lambda_{l,i}^{\rm ns}$ 
for the two maxima in \eqref{eq:state_constraints_condition_max}
to rewrite it as 
\begin{subequations}\label{eq:state_constraints_condition_epi}
    \begin{align}
        \lambda_{l,i}^{\rm sm} 
        + \lambda_{l,i}^{\rm ns} 
        + \eta^{\rm sm}_{l,i}\|z_{l-1}\|_* 
        + \eta^{\rm ns}_{l,i} \|\Delta u_{l}\|_*  \leq b_i,
    \\
    s_{l, i}^{\nu\; \intercal} z_{l-1} \leq \lambda_{l,i}^{\rm sm},
    \\
    n_{l, i}^{\nu\; \intercal} \Delta u_{l}  \leq \lambda_{l,i}^{\rm ns},
    \end{align}
\end{subequations}
for $\nu=0,\ldots, q$, $i=1,\ldots, n_c$, and $l\geq 1$.

Likewise, for $l=0$ suppose $\hat\psi_{0,i}$ is $L_{0,i}$-smooth 
and let $\eta_{0,i}=\nicefrac{L_{0,i}}{8q}$. 
Then the constraint $H_i x(t)\leq b_i$ holds on $I_0$
whenever 
$\delta^*_{\mathcal Q_{0,i}}(z_0) {}={}
\max_{\nu=0,\ldots,q}p_{0,i}^{\nu\;\intercal}z_0
+
\eta_{0,i}\|z_0\|_*
\leq b_i$.
Equivalently, introducing a slack variable $\lambda_{0,i}$, this can be written as
\begin{subequations}\label{eq:state_constraints_condition_epi:l0}
    \begin{align}
        \lambda_{0,i}+\eta_{0,i}\|z_0\|_* \leq{}& b_i,\\
        p_{0,i}^{\nu\;\intercal}z_0 \leq{}& \lambda_{0,i},
    \end{align}
\end{subequations}
for $\nu=0,\ldots,q$ and $i=1, \ldots, n_c$.

If the norm used in Lemma \ref{lem:bernstein_bound} 
is the infinity norm, then the dual norm in 
\eqref{eq:state_constraints_condition_max} and 
\eqref{eq:state_constraints_condition_epi:l0} is the 1-norm, thus
by introducing standard epigraphical reformulations 
for the terms $\|z_{l-1}\|_1$, 
$\|\Delta u_l\|_1$, and $\|z_0\|_1$, the tightened state 
constraints can be written as polyhedral constraints.
If $\mathcal{U}$ is polyhedral and the cost is quadratic convex, 
the resulting approximation is a quadratic program.

\subsection{Tighter approximations}\label{sec:tighter_approximations}
The fact that we treat $\psi^{\rm sm}$ and $\psi^{\rm ns}$ separately
introduces some conservatism. To address this issue we subdivide 
$[0, 1]$ into equal subintervals 
$J_r=[\nicefrac{r}{N_s}, \nicefrac{r+1}{N_s}]$ 
for $r=0,\ldots, N_s-1$ and determine Bernstein enclosures 
over each subinterval. 
We will show that the overall conservatism can be controlled by $N_s$.
This approach is somewhat reminiscent of composite Bernstein
polynomials \citep{MacLin+2024}, which serve as a 
collocation method for optimal control problems.
Instead, here we do not approximate the trajectories by 
Bernstein polynomials, but use them for the construction 
of enclosures.

% What is the gap?
For fixed $l\geq 1$, $i=1, \ldots, n_c$, $z_{l-1}$, 
and $\Delta u_l$, let $y_l=(z_{l-1},\Delta u_l)$. 
The exact left-hand side of the $i$th constraint is
\begin{equation}
    M_{l,i}(y_l)
    {}\dfn{} 
    \sup_{\tau \in [0,1]}
    \{
        \psi_{l, i}^{\rm sm}(\tau)^\intercal z_{l-1}
        {}+{}
        \psi_{l, i}^{\rm ns}(\tau)^\intercal \Delta u_{l}
    \}.
    \label{eq:M}
\end{equation}
The corresponding exact constraint is $M_{l,i}(y_l)\leq b_i$.
By applying the approach of Section \ref{sec:polytopic_overapproximations}
let the local Bernstein enclosure of the smooth part on $J_r$ be
$
    \mathcal Q_{l,i,r}^{\rm sm}=
    \conv\{s_{l,i,r}^{\nu}\}_{\nu=0}^{q}
    \oplus
    \mathcal B_{\eta_{l,i,r}^{\rm sm}},
$
and let the local Bernstein enclosure of the nonsmooth part be
$
    \mathcal Q_{l,i,r}^{\rm ns}
    =
    \conv\{n_{l,i,r}^{\nu}\}_{\nu=0}^{q}
    \oplus
    \mathcal B_{\eta_{l,i,r}^{\rm ns}},
$
such that 
$\psi_{l,i}^{\rm sm}(J_r)\subseteq \mathcal Q_{l,i,r}^{\rm sm},$
and 
$\psi_{l,i}^{\rm ns}(J_r)\subseteq \mathcal Q_{l,i,r}^{\rm ns}$.
The Bernstein enclosure of the left-hand side over $J_r$ is
\begin{align}
    &B_{l,i,r,N_s,q}(y_l)
    {}\dfn{}
    \hspace{-0.5em}
    \max_{\nu=0,\ldots, q} s_{l, i, r}^{\nu\; \intercal}z_{l-1}
    + \eta^{\rm sm}_{l, i, r}\|z_{l-1}\|_*
    \notag\\
    &\qquad{}+ \max_{\nu=0,\ldots, q} n_{l, i, r}^{\nu\; \intercal}\Delta u_{l}
    + \eta^{\rm ns}_{l, i, r}\|\Delta u_{l}\|_* .
    \label{eq:Br}
\end{align}
The corresponding sufficient constraint is
$B_{l,i,r,N_s,q}(y_l)\leq b_i$.
For the fixed data in this subsection, write
$M\dfn M_{l,i}(y_l)$ and $B_r\dfn B_{l,i,r,N_s,q}(y_l)$.
We define the \emph{gap} between $B_r$ and $M$ as 
\begin{equation}
    \gap_{N_s} \dfn \max_{r} B_r - M.
\end{equation}
Next, we will show that $\gap_{N_s} = O(N_s^{-\min\{1, 2\alpha\}})$.

% -----------------------------------
% THEOREM 3
% Conservatism estimate
% -----------------------------------
% \begin{framed}
\begin{theorem}[Conservatism estimate]%
\label{thm:conservatism_estimate}
    Fix $l \geq 1$, $i=1, \ldots, n_c$, $z_{l-1}$, 
    and $\Delta u_l$. Suppose 
    (i) $\hat{\psi}^{\rm ns}_{l, i}$ is $L^{\rm ns}_{l, i}$-smooth, 
    and 
    (ii) $\psi^{\rm sm}_{l, i}$ is $L^{\rm sm}_{l, i}$-smooth.
    Let $K^{\rm sm}_{l,i}$ be a Lipschitz constant of 
    $\psi^{\rm sm}_{l,i}$ on $[0,1]$.
    Then, for a uniform partition of $[0,1]$ into $N_s$ subintervals and 
    fixed Bernstein degree $q$, $\mathrm{gap}_{N_s} \geq 0$ and
    \begin{equation}
        \mathrm{gap}_{N_s} \leq
        \left(\tfrac{K^{\rm sm}_{l,i}}{N_s}
        +
        \tfrac{L^{\rm sm}_{l, i}}{8qN_s^2}
        \right)
        \|z_{l-1}\|_*
        +
        \frac{L^{\rm ns}_{l, i}\|\Delta u_l\|_*}{8qN_s^{2\alpha}}.
    \end{equation}    
\end{theorem}
% \end{framed}

\begin{proof}
% Proof - part I
For notational convenience, let us define 
$a(\tau)=\psi_{l,i}^{\rm sm}(\tau)^\intercal z_{l-1}$ and
$c(\tau)=\psi_{l,i}^{\rm ns}(\tau)^\intercal \Delta u_{l}$.
Since $\psi^{\rm sm}_{l,i}$ is $K^{\rm sm}_{l,i}$-Lipschitz, 
$a$ is $K_a$-Lipschitz with 
$K_a\dfn K^{\rm sm}_{l,i}\|z_{l-1}\|_*$.
For each subinterval $J_r$, define
$M_r {}={} \sup_{\tau\in J_r}\{a(\tau)+c(\tau)\}$.
Since the $J_r$ form a partition of $[0,1]$, we have
$M=\max_{r} M_r$.
For $\tau\in J_r$ note that 
$a(\tau) \leq \delta^*_{\mathcal Q_{l,i,r}^{\rm sm}}(z_{l-1})$
and 
$c(\tau) \leq \delta^*_{\mathcal Q_{l,i,r}^{\rm ns}}(\Delta u_l)$,
therefore,
\begin{equation}
    a(\tau)+c(\tau)\leq B_r,
    \forall \tau\in J_r,
\end{equation}
where $B_r$ is given in \eqref{eq:Br}.
Taking the supremum over $J_r$ gives $M_r\leq B_r$, so
$M=\max_r M_r\leq \max_r B_r$, which implies 
$\gap_{N_s}\geq 0$.

% Proof - part II
Let $A_r{}={} \sup_{\tau\in J_r}a(\tau)$ and 
$C_r{}={} \sup_{\tau\in J_r}c(\tau)$.
Since the vertices $s_{l,i,r}^{\nu}$ and $n_{l,i,r}^{\nu}$ are
sampled values of $\psi_{l,i}^{\rm sm}$ and $\hat{\psi}_{l,i}^{\rm ns}$
on $J_r$, respectively,
\begin{subequations}
    \begin{align}
        \max_{\nu=0,\ldots,q}
        (s_{l,i,r}^{\nu})^\intercal z_{l-1}
        \leq A_r,
        \\
        \max_{\nu=0,\ldots,q}
        (n_{l,i,r}^{\nu})^\intercal \Delta u_l
        \leq C_r.
    \end{align}
\end{subequations}
From \eqref{eq:Br} it follows that
\begin{equation}
    B_r
    \leq
    A_r+C_r
    +
    \eta_{l,i,r}^{\rm sm}\|z_{l-1}\|_*
    +
    \eta_{l,i,r}^{\rm ns}\|\Delta u_l\|_* .
    \label{eq:Br_upper_bound}
\end{equation}

Because $a$ and $c$ are continuous on the compact interval $J_r$,
there exist $\tau_a,\tau_c\in J_r$ such that
$a(\tau_a)=A_r$ and $c(\tau_c)=C_r$.
Therefore, using the Lipschitz continuity of $a$,
\begin{align}
    &M_r
    {}={}
    \sup_{\tau\in J_r}\{a(\tau)+c(\tau)\}
    \geq
    a(\tau_c)+C_r
    \notag\\
    \Rightarrow{}& 
    A_r+C_r-M_r
    \leq
    A_r-a(\tau_c)
    =
    a(\tau_a)-a(\tau_c)
    \notag\\
    &\qquad \leq
    K_a|\tau_a-\tau_c|
    \leq
    K_a|J_r|
    =
    \frac{K_a}{N_s}
    \label{eq:Ar_Cr_bound}
\end{align}
Combining \eqref{eq:Br_upper_bound} and \eqref{eq:Ar_Cr_bound} we obtain
\begin{equation}
    B_r-M_r
    \leq
    \frac{K^{\rm sm}_{l,i}\|z_{l-1}\|_*}{N_s}
    +
    \eta_{l,i,r}^{\rm sm}\|z_{l-1}\|_*
    +
    \eta_{l,i,r}^{\rm ns}\|\Delta u_l\|_* .
\end{equation}
To bound $\eta_{l,i,r}^{\rm sm}$ we note that the maps
$\tau {}\mapsto{} \psi_{l,i}^{\rm sm}\left(\frac{r+\tau}{N_s}\right),$
are $L^{\rm sm}_{l, i}/N_s^2$-smooth on $[0, 1]$, so by
Lemma~\ref{lem:bernstein_bound},
\begin{equation}
    \eta_{l,i,r}^{\rm sm}
    \leq
    \frac{L^{\rm sm}_{l, i}}{8qN_s^2}.
\end{equation}

For the nonsmooth part, we use the fractional time $\rho=\tau^\alpha$
via which the interval $J_r$ is mapped to
\begin{equation}
    J_r^{\alpha}
    =
    \left[
        \left(\frac{r}{N_s}\right)^\alpha,
        \left(\frac{r+1}{N_s}\right)^\alpha
    \right].
\end{equation}
Since $\tau\mapsto\tau^\alpha$ is concave on $[0,1]$, 
the largest such
length is attained at $r=0$, and therefore
$|J_r^{\alpha}| {}\leq{} N_s^{-\alpha}$.
The normalized local map for the nonsmooth part has second derivative
bounded by $L^{\rm ns}_{l, i}|J_r^{\alpha}|^2$. Again by
Lemma~\ref{lem:bernstein_bound},
\begin{equation}
    \eta_{l,i,r}^{\rm ns}
    \leq
    \frac{L^{\rm ns}_{l, i}|J_r^{\alpha}|^2}{8q}
    \leq
    \frac{L^{\rm ns}_{l, i}}{8qN_s^{2\alpha}}.
\end{equation}
Therefore, for every $r$,
\begin{multline}
    B_r-M_r
    \leq
    \frac{K^{\rm sm}_{l,i}\|z_{l-1}\|_*}{N_s}
    \\
    +
    \frac{L^{\rm sm}_{l, i}\|z_{l-1}\|_*}{8qN_s^2}
    +
    \frac{L^{\rm ns}_{l, i}\|\Delta u_l\|_*}{8qN_s^{2\alpha}}.
\end{multline}
Observing that
$
    \gap_{N_s}
    =
    \max_r B_r-\max_r M_r
    \leq
    \max_r(B_r-M_r),$
the assertion follows.
\end{proof}

For the case $l=0$, as discussed in 
Section \ref{sec:polytopic_overapproximations}, 
$\hat{\psi}_{0,i}$ depends only on the fractional time 
$\rho=\tau^\alpha \in [0,1]$. We subdivide the $[0, 1]$
into $N_s$ equal intervals and analogously to \eqref{eq:M}
and \eqref{eq:Br} we define 
$M_{0,i} = \sup_{\rho \in [0, 1]}\hat{\psi}_{0,i}(\rho)^\intercal z_0$
and $B_{0,i,r}$. 
The gap at $l=0$ is $\gap_{0,i,N_s} = \max_{r}B_{0,i,r} - M_{0,i}$.
Following the same arguments as in the above proof, we arrive at 
\begin{equation}
        0\leq \gap_{0,i,N_s}
        \leq
        \frac{L_{0,i}\|z_0\|_*}{8qN_s^2},
\end{equation}
that is, $\gap_{0,i,N_s} = O(N_s^{-2})$.

Next, we show that under a Slater-type constraint qualification,
as $N_s\to\infty$ the sequence of global feasible sets obtained 
from the Bernstein enclosures converges to the feasible set of 
the original semi-infinite constraints.

% -----------------------------------
% THEOREM 4
% Convergence of feasible sets
% -----------------------------------
% \begin{framed}
\begin{theorem}[Convergence of feasible sets]
\label{thm:PK_convergence_feasible_sets}
Fix a prediction horizon $N$ and a Bernstein degree $q$.
Let $w=(x_0,u_0,\ldots,u_{N-1})$. 
For $l=0$, let $y_0(w)=(x_0,u_0)$, and for
$l=1,\ldots,N-1$, let
$y_l(w)=(z_{l-1},\Delta u_l)$, where
$z_{l-1}=(x_0,u_0,\ldots,u_{l-1})$ and
$\Delta u_l=u_l-u_{l-1}$.
For $l\geq1$, let $M_{l,i}$ and $B_{l,i,r,N_s,q}$ be as in
\eqref{eq:M} and \eqref{eq:Br}, and set
\[
    B_{l,i,N_s,q}(y_l){}\dfn{}
    \max_{r=0,\ldots,N_s-1}B_{l,i,r,N_s,q}(y_l).
\]
For $l=0$, define $M_{0,i}$ and $B_{0,i,r}$ as above, using the
fractional-time representation, and set
\[
    B_{0,i,N_s,q}(y_0){}\dfn{}
    \max_{r=0,\ldots,N_s-1}B_{0,i,r}(y_0).
\]
Let $\mathcal I=\{0,\ldots,N-1\}\times\{1,\ldots,n_c\}$ and
define the exact and approximate feasible sets
\begin{subequations}
    \begin{align}
        \F
        {}={}&
        \left\{
            w {:}
            M_{l,i}(y_l(w))\leq b_i,\; 
            (l,i)\in\mathcal{I}
        \right\},
        \\
        \F_{N_s}
        {}={}&
        \left\{
            w {:}
            B_{l,i,N_s,q}(y_l(w))\leq b_i,\;
            (l,i)\in\mathcal{I}
        \right\}.
        \label{eq:FNs}
    \end{align}
\end{subequations}
Suppose that the assumptions of Theorem~\ref{thm:conservatism_estimate}
hold for all $l=1,\ldots,N-1$ and $i=1,\ldots,n_c$, 
and that the analogous assumptions for $l=0$ hold.
Assume also that there is a globally strictly feasible point 
$\bar{w}$, i.e.,
\begin{equation}
    M_{l,i}(y_l(\bar{w}))<b_i,
    (l,i)\in\mathcal{I}.
    \label{eq:strict_feasibility_condition}
\end{equation}
Then,
\begin{equation}
    \lim_{N_s \to \infty}\F_{N_s} = \F,
\end{equation}
where the convergence is in 
the Painlev{\'e}-Kuratowski sense~\citep{RocWet98}.
\end{theorem}
% \end{framed}

\begin{proof}
    % 1. Intro
    By Theorem~\ref{thm:conservatism_estimate} and its analogue for $l=0$,
    $0\leq B_{l,i,N_s,q}(y_l)-M_{l,i}(y_l)$ for every 
    $(l,i)\in\mathcal I$, $N_s$, and $y_l$.
    If $w\in\F_{N_s}$, then
    $M_{l,i}(y_l(w))\leq B_{l,i,N_s,q}(y_l(w))\leq b_i$
    for all $l$ and $i$, so $w\in\F$.
    Therefore, $\F_{N_s}\subseteq\F$ for every $N_s$.

    The set $\F$ is closed because, by Berge's maximum 
    theorem~\citep[Sec.~VI-3]{Berge97}, each $M_{l,i}$ is continuous
    as the supremum over a compact interval of functions that are 
    continuous in time and affine in $y_l$, and $y_l$ is an affine 
    function of $w$.

    By the definition of the outer limit
    \citep[Def.~4.1]{RocWet98}, any point in
    $\limsup_{N_s\to\infty}\F_{N_s}$ is the limit of a sequence
    $\zeta_{N_s}\in\F_{N_s}$ along an infinite index set. Since
    $\zeta_{N_s}\in\F$ and $\F$ is closed, that limit belongs to $\F$,
    that is,
    \begin{equation}
        \limsup_{N_s\to\infty}\F_{N_s}\subseteq\F .
        \label{eq:PK_limsup}
    \end{equation}

    % 2. Hit and miss criterion; here we prove that \F is contained in 
    %    the liminf \F_{N_s}   
    Using the hit-and-miss criterion \citep[Thm.~4.5(a)]{RocWet98} we 
    will show that $\F \subseteq \liminf_{N_s\to\infty} \F_{N_s}$.
    Let $O$ be an open set such that $\F\cap O {}\neq{} \emptyset$.
    Take a $w^* {}\in{} \F\cap O$. 
    By the strict feasibility condition in 
    \eqref{eq:strict_feasibility_condition}, 
    there exists $\bar{w}$ such that
    $M_{l,i}(y_l(\bar{w}))<b_i,$
    for $(l,i)\in\mathcal I$.
    Define
    \begin{equation}
        \sigma
        {}={}
        \min_{(l,i)\in\mathcal I}
        \left\{b_i-M_{l,i}(y_l(\bar{w}))\right\}
        {}>{} 0.
    \end{equation}
    For $\theta\in(0,1)$, set
    $w_\theta=\theta \bar{w} {}+{} (1-\theta)w^*$.
    Since $w_\theta\to w^*$ as $\theta{}\downarrow{}0$, 
    and since $O$ is open, we
    may choose $\theta$ sufficiently small so that
    $w_\theta\in O$.
    Note that each map $w\mapsto y_l(w)$ is affine and, 
    from \eqref{eq:M} and its analogue for $l=0$, 
    $M_{l,i}$ is the supremum over time of affine 
    functions of $y_l$, so it is convex.
    For every $(l,i)\in\mathcal I$, by the convexity of $M_{l,i}$,
    \begin{align}
        M_{l,i}(y_l(w_\theta))
        {}\leq{}&
        (1-\theta)M_{l,i}(y_l(w^*))
        +\theta M_{l,i}(y_l(\bar{w}))  
        \notag\\
        {}\leq{}&
        \theta(b_i-\sigma) + (1-\theta)b_i  {}={}
        b_i-\theta\sigma.
    \end{align}
    Thus, $w_\theta$ is strictly feasible for the exact constraints, 
    with margin at least $\theta\sigma$. It remains to show that 
    $w_\theta$ is eventually contained in all $\F_{N_s}$.

    By Theorem~\ref{thm:conservatism_estimate} and its analogue for $l=0$,
    for each fixed $w_\theta$ it is
    $B_{l,i,N_s,q}(y_l(w_\theta))
    \to M_{l,i}(y_l(w_\theta))$ as $N_s\to\infty$.
    Since the number of intervals and constraint rows is finite, 
    there exists $N_0$ such that,
    for all $N_s\geq N_0$,
    \begin{align}
        B_{l,i,N_s,q}(y_l(w_\theta))
        {}\leq{}&
        M_{l,i}(y_l(w_\theta))+\frac{\theta\sigma}{2} 
        \notag
        \\
        {}\leq{}&
        b_i-\frac{\theta\sigma}{2}
        {}<{}
        b_i,
    \end{align}
    for all $(l,i)\in\mathcal I$.
    Therefore, $w_\theta\in \F_{N_s}$
    for all $N_s\geq N_0$.
    Since also $w_\theta\in O$, it is
    $O {}\cap{} \F_{N_s} {}\neq{} \emptyset$ for
    $N_s {}\geq{} N_0$, so the hit-and-miss condition holds.
    In light of \eqref{eq:PK_limsup}, this completes the proof.
\end{proof}

\begin{remark}[Input constraints]
Theorem~\ref{thm:PK_convergence_feasible_sets} is stated only for the
state constraints. If input constraints $u_k\in\mathcal{U}$ are also
included, with $\mathcal{U}$ closed and convex, define
\begin{align*}
    \mathcal{W}_N
    {}\dfn{} \bigl\{(x_0,u_0,\ldots,u_{N-1}) : {}&
    u_k\in\mathcal{U},
    \\
    &k=0,\ldots,N-1\bigr\}.
\end{align*}
The corresponding exact and approximate feasible sets are
$\F\cap\mathcal{W}_N$ and $\F_{N_s}\cap\mathcal{W}_N$.
The same proof applies, provided the strictly feasible point $\bar{w}$
also belongs to $\mathcal{W}_N$.
Indeed, since $\mathcal{U}$ is closed, $\mathcal{W}_N$ is 
closed, so the outer-limit argument holds, and
$w_\theta$ remains feasible because $\mathcal{U}$ is convex.
\end{remark}

Theorem \ref{thm:PK_convergence_feasible_sets} shows that 
the conservatism introduced by the Bernstein tightening 
can be mitigated by sufficiently large $N_s$.
Moreover, under standard continuity and compactness assumptions 
on the cost and the set of minimizers, this feasible-set convergence 
can be used to establish convergence of optimal values and outer 
convergence of optimizers to solutions of the original semi-infinite OCP,
as in \citep{SopPatSar13}.

% Similar to \cite{SopPatSa13}, the convergence of the feasible 
% sets $\F_{N_s}$ to $\F$ implies that for every $x_0$ the 
% corresponding optimal values converge, and the obtained 
% optimizers converge to an optimizer of the original 
% semi-infinite problem.

\section{Simulation results}

\subsection{Tightness of Bernstein enclosures}\label{sec:simulations:tightness}
Firstly, we want to assess the tightness of Bernstein enclosures
by looking at how the gap scales with the number of subintervals,
$N_s$, and the Bernstein degree, $q$.
To this end, we consider a linear fractional system of the form 
\eqref{eq:flti} with
\begin{equation}
    A = \begin{bmatrix}
        \phantom{-}0 & \phantom{-}1\phantom{.4} \\  -2 & -0.4
    \end{bmatrix},
    B = \begin{bmatrix}
        0 \\ 1
    \end{bmatrix},
\end{equation}
and $h=0.2$, 
$x_0=(0.2, -0.1)$,
$u_0=0.3$, 
$u_1 = 0.1$,
$u_2 = 0.25$,
$H=[1 ~ 0]$.
We study the gap for different values of $\alpha$.

\begin{figure}
    \centering
    \tikzsetnextfilename{gap}
\begin{tikzpicture}
\begin{semilogyaxis}[
    name=gap1,
    width=\GapPlotWidth,
    height=\ShortPlotHeight,
    scale only axis,
    trim axis left,
    trim axis right,
    xlabel={Number of subintervals $N_s$},
    ylabel={Gap},
    grid=both,
    xtick={1,2,4,8,16,32},
    xmin=1,
    xmax=32,
    major grid style={draw=gray!20},
    minor grid style={draw=gray!20},
    legend columns=2,
    legend style={
        at={(0.5,1.05)},
        anchor=south,
        draw=none,
        fill=none,
        /tikz/every even column/.append style={column sep=0.8em}
    },
    mark size=1.5pt,
]
\addplot+[mark=o] coordinates {
    (1,0.0171773591764748)
    (2,0.0076765778818226)
    (4,0.0028998159037785)
    (8,0.0013987860947784)
    (16,0.0006770055536653)
    (32,0.0003325706027195)
};
\addlegendentry{$\alpha=0.35$}

\addplot+[mark=square] coordinates {
    (1,0.0240342788596745)
    (2,0.0130576547029798)
    (4,0.0070908086850799)
    (8,0.0036671557989282)
    (16,0.0018836543403132)
    (32,0.0009589343079362)
};
\addlegendentry{$\alpha=0.5$}

\addplot+[mark=triangle] coordinates {
    (1,0.0100666603141316)
    (2,0.0035999831259991)
    (4,0.0013182144141904)
    (8,0.0004661898335634)
    (16,0.0001645626330301)
    (32,0.0000583490512935)
};
\addlegendentry{$\alpha=0.75$}

\addplot+[mark=diamond] coordinates {
    (1,0.0035694936243882)
    (2,0.0009209457994039)
    (4,0.0002407499513841)
    (8,0.0000584867675518)
    (16,0.0000147428861602)
    (32,0.0000038191917359)
};
\addlegendentry{$\alpha=1.0$}
\end{semilogyaxis}

\begin{semilogyaxis}[
    name=gap2,
    at={(gap1.south)},
    anchor=north,
    yshift=-1.25cm,
    width=\GapPlotWidth,
    height=\ShortPlotHeight,
    scale only axis,
    trim axis left,
    trim axis right,
    xlabel={Bernstein degree $q$},
    ylabel={Gap},
    grid=both,
    xmin=2,
    xmax=21,
    major grid style={draw=gray!30},
    minor grid style={draw=gray!10},
    legend style={
        at={(0.98,0.98)},
        anchor=north east,
        fill=white,
        draw=gray
    },
    mark size=1.5pt,
]
\addplot+[mark=o] coordinates {
    (3,0.0014154454328237)
    (5,0.0014042233226166)
    (10,0.0013987860947784)
    (20,0.0013964183835912)
    % (40,0.0013954848979830)
};
% \addlegendentry{$\alpha=0.35$}

\addplot+[mark=square] coordinates {
    (3,0.0040749910922798)
    (5,0.0038482394486234)
    (10,0.0036671557989282)
    (20,0.0035766141860606)
    % (40,0.0035313433796264)
};
% \addlegendentry{$\alpha=0.5$}

\addplot+[mark=triangle] coordinates {
    (3,0.0005145115627391)
    (5,0.0004868008680687)
    (10,0.0004661898335634)
    (20,0.0004558843163108)
    % (40,0.0004507349296285)
};
% \addlegendentry{$\alpha=0.75$}

\addplot+[mark=diamond] coordinates {
    (3,0.0000832251043750)
    (5,0.0000694554359640)
    (10,0.0000584867675518)
    (20,0.0000529729541824)
    % (40,0.0000502160559692)
};
% \addlegendentry{$\alpha=1.0$}
\end{semilogyaxis}
\end{tikzpicture}
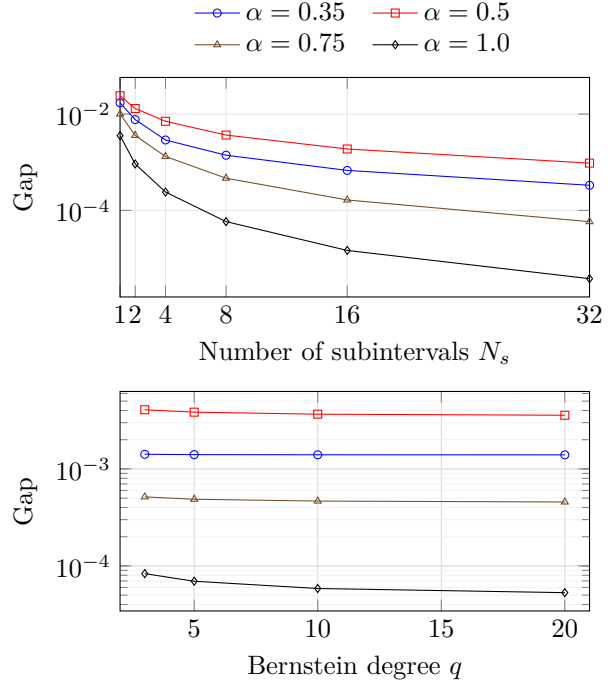
    \caption{Gap at $l=2$.
        (Top) Gap vs $N_s$ for different values of $\alpha$ with fixed $q=10$,
        (Bottom) Gap at different $q$ with fixed $N_s=8$.}
    \label{fig:gap_sweep}
\end{figure}

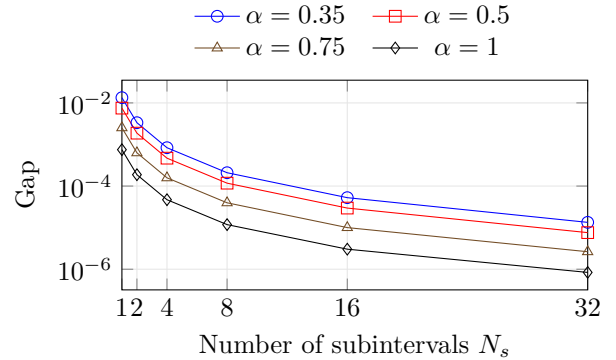
\begin{figure}
    \centering
    \tikzsetnextfilename{gap-l0}
\begin{tikzpicture}
\begin{semilogyaxis}[
    width=\WidePlotWidth,
    height=\WidePlotHeight,
    xlabel={Number of subintervals $N_s$},
    ylabel={Gap},
    grid=both,
    major grid style={draw=gray!20},
    minor grid style={draw=gray!20},
    xtick={1,2,4,8,16,32},
    xmin=1,
    xmax=32,
    legend columns=2,
    legend style={
        at={(0.5,1.05)},
        anchor=south,
        draw=none,
        fill=none,
        /tikz/every even column/.append style={column sep=0.8em}
    },
    mark size=2.2pt,
]
\addplot+[mark=o] coordinates {
    (1,1.3365163532994500e-02)
    (2,3.3430780359175998e-03)
    (4,8.3621355465940005e-04)
    (8,2.0933018877960000e-04)
    (16,5.2588481785798002e-05)
    (32,1.3400448733352777e-05)
};
\addlegendentry{$\alpha=0.35$}

\addplot+[mark=square] coordinates {
    (1,7.4957109588720001e-03)
    (2,1.8746978393396000e-03)
    (4,4.6894539036609998e-04)
    (8,1.1744503307980001e-04)
    (16,2.9562173200908811e-05)
    (32,7.5904867462839309e-06)
};
\addlegendentry{$\alpha=0.5$}

\addplot+[mark=triangle] coordinates {
    (1,2.5225951088758002e-03)
    (2,6.3086633012210001e-04)
    (4,1.5785650382940000e-04)
    (8,3.9594353160199443e-05)
    (16,1.0027606417606361e-05)
    (32,2.6357660085907231e-06)
};
\addlegendentry{$\alpha=0.75$}

\addplot+[mark=diamond] coordinates {
    (1,7.5001553901209996e-04)
    (2,1.8759433340750001e-04)
    (4,4.6979178586537003e-05)
    (8,1.1824160166762354e-05)
    (16,3.0352511042608388e-06)
    (32,8.3800432287839399e-07)
};
\addlegendentry{$\alpha=1$}

\end{semilogyaxis}
\end{tikzpicture}
    \caption{Gap against $N_s$ for $l=0$ using fixed Bernstein 
    degree $q=10$. The fast rate of convergence agrees with the 
    theoretical $O(1/N_s^2)$ of Theorem 
    \ref{thm:conservatism_estimate}.}
    \label{fig:l0_gap_vs_Ns}
\end{figure}

As illustrated in Figues \ref{fig:gap_sweep}
and \ref{fig:l0_gap_vs_Ns},
for fixed Bernstein degree $q$, the gap 
decreases with $N_s$ at a rate which is 
consistent with the theoretical worst-case estimate.
Increasing $q$ for fixed $N_s$ yields only 
modest improvements, suggesting that the main source 
of conservatism in this example is the separate 
treatment of the smooth and nonsmooth terms rather 
than the Bernstein approximation error.
The lowest gap values are observed for values of the 
fractional order, $\alpha$, closer to 1. 
The results shown in Fig.~\ref{fig:gap_sweep}
demonstrate that subdivision is the 
dominant mechanism for reducing conservatism, 
confirming the bound of Theorem \ref{thm:conservatism_estimate}.

\subsection{Computational complexity}\label{sec:simulations:complexity}
In this section, we solve a sampled-data fractional
optimal control problem and measure the solution time for different
prediction horizons $N$ and numbers of subintervals $N_s$. 
We use the same system as in Section \ref{sec:simulations:tightness}, 
with $\alpha=0.75$, $q=10$, and $h=0.2$.
The total cost is
\begin{align}
    J ={}& \int_{0}^{Nh}
    \bigl(x(\xi)^\intercal Qx(\xi)
    +u(\xi)^\intercal Ru(\xi)\bigr)\dd \xi
    \notag\\
    &{}+ x(Nh)^\intercal x(Nh),
    \label{eq:total_cost_example}
\end{align}
where $Q=\operatorname{diag}(6, 0.5)$, and $R=0.01$.
The cost function is discretized using the 
trapezoidal rule.
The system is subject to state constraints
$\smallmat{-0.55\\-1.5} \leq x(t) \leq \smallmat{0.55\\1.5}$, 
and the input bounds $|u_k|\leq 8$.
The numerically stable method of Garrappa~\cite{Garrappa2015} is used
for the computation of Mittag-Leffler functions.

The proposed methodology introduces 
$O(n_cNN_sq)$ constraints of the type 
\eqref{eq:state_constraints_condition_epi} 
and \eqref{eq:state_constraints_condition_epi:l0}.
The resulting inequalities are sparse and have a causal block-staircase
structure. Constraints on interval $I_l$ involve only $x_0$ and the input 
history $u_0,\ldots,u_l$. 
Thus, increasing $N_s$ or $q$ adds rows with repeated 
sparsity patterns rather than destroying the lower-triangular structure 
induced by the finite-horizon dynamics.
The resulting large-but-sparse optimization problems can, therefore, 
be solved efficiently.

Fig.~\ref{fig:runtime_vs_horizon} shows the solve times of 
the sampled-data fractional OCP against $N$ for different 
values of $N_s$.
We use $q=10$ and the initial state $x_0=(-0.4,-0.6)$.
The problems are solved with OSQP~\citep{osqp2020} using Python v3.13 
on an Apple M4 Pro laptop with 24 GB of RAM.

\begin{figure}
\centering
\tikzsetnextfilename{solve-times}
\begin{tikzpicture}
\begin{axis}[
    width=\PlotWidth,
    height=\MediumPlotHeight,
    xlabel={Prediction horizon, $N$},
    ylabel={Solve time ($\unit{ms}$)},
    y label style={at={(-0.14,0.5)}},
    xmin=1, xmax=21,
    ymin=0, ymax=65,
    grid=both,
    major grid style={draw=gray!20},
    minor grid style={draw=gray!20},
    legend columns=2,
    legend style={
        draw=gray!50,
        fill=white, 
        fill opacity=0.65, 
        text opacity=1},
    legend pos=north west,
    tick align=outside,
    mark size=1.6pt,
]
\addplot+[
    mark=o,
    color=gray
] coordinates {
(2,0.227999873459)
(3,0.366041902453)
(4,0.570749863982)
(5,0.771125080064)
(6,1.0493749287)
(7,1.27470912412)
(8,1.56520796008)
(9,1.92341697402)
(10,2.20045796596)
(11,2.4227919057)
(12,2.8007500805)
(13,3.20058292709)
(14,3.45916696824)
(15,3.81908402778)
(16,4.26570791751)
(17,4.61683305912)
(18,5.02733304165)
(19,5.38629083894)
(20,5.84991695359)
};
\addlegendentry{$N_s=2$}

\addplot+[
    color=orange!85!black,
    mark=square,
] coordinates {
(2,0.422124983743)
(3,0.886583002284)
(4,1.37845799327)
(5,1.9482919015)
(6,2.44029215537)
(7,2.99241719767)
(8,3.60345817171)
(9,4.16570785455)
(10,4.79987519793)
(11,5.4685829673)
(12,6.27620797604)
(13,7.04066711478)
(14,7.62145896442)
(15,8.53745802306)
(16,9.44566703402)
(17,10.3345829993)
(18,11.2290410325)
(19,12.2697080951)
(20,13.3132499177)
};
\addlegendentry{$N_s=4$}

\addplot+[
    color=green!55!black,
    mark=triangle,
] coordinates {
(2,0.866333022714)
(3,1.84879195876)
(4,3.85404122062)
(5,3.67937516421)
(6,4.64491685852)
(7,5.80425001681)
(8,6.78083300591)
(9,7.86437513307)
(10,9.04729100876)
(11,10.3520830162)
(12,11.667124927)
(13,13.0221250001)
(14,15.0019580033)
(15,17.0287501533)
(16,18.8012919389)
(17,20.4689998645)
(18,21.9787498936)
(19,23.9378330298)
(20,25.9002079256)
};
\addlegendentry{$N_s=8$}

\addplot+[
    color=red!70!black,
    mark=diamond,
] coordinates {
(2,1.85583299026)
(3,3.87933291495)
(4,7.70154199563)
(5,10.0214171689)
(6,9.23862494528)
(7,11.142665986)
(8,15.4964171816)
(9,18.142499961)
(10,20.6991669256)
(11,20.5556249712)
(12,26.8049580045)
(13,29.972959077)
(14,34.6165420488)
(15,38.1549589802)
(16,41.8325420469)
(17,45.4274159856)
(18,48.9397500642)
(19,53.6752918269)
(20,57.8358341008)
};
\addlegendentry{$N_s=16$}

\end{axis}
\end{tikzpicture}
\caption{Solve time (average over 20 runs) against prediction horizon 
for the Bernstein-tightened sampled-data fractional optimal control 
problems.}
\label{fig:runtime_vs_horizon}
\end{figure}
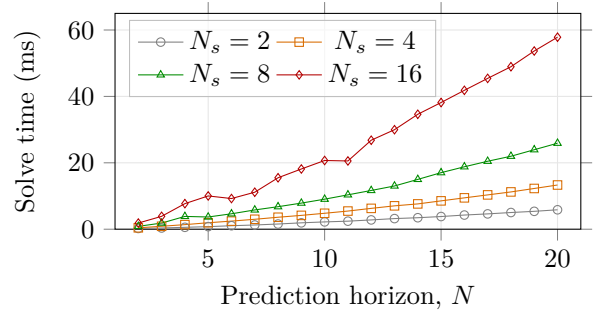

Fig.~\ref{fig:runtime_vs_horizon} shows that the solve times grow roughly 
linearly with the horizon length for fixed $N_s$, and roughly proportionally 
with the number of subintervals. 
Note that even with the tightest formulation with $N_s=16$
the solve times are below  $\unit[60]{ms}$ at $N=20$.

In addition, we are interested in how the computation times 
scale with the number of states, $n$. To this end, 
we benchmark against systems where $A \in \R^{n\times n}$ 
is a tridiagonal 
matrix,
\begin{equation}
    A = \begin{bmatrix}
       -1 & 0.2\\
        0.2 & -1 & 0.2 \\
        &\ddots&\ddots&\ddots \\ 
        &&0.2 & -1 & 0.2 \\
        &&&0.2 & -1 
    \end{bmatrix},
\end{equation}
and $B = [1 ~ 0 ~ \ldots ~ 0 ~ 1]^\intercal$, subject to the 
state constraints $\|x(t)\|_{\infty} \leq 2$ and 
$|u_k| \leq 8$. We set $\alpha = 0.75$,
$h=0.2$, $N_s=8$, $q=10$, and use the cost function in 
\eqref{eq:total_cost_example}.
For values of $n$ between 2 and 40 and horizons $N=10$ and $N=20$
we solve 1000 problems with the initial condition, $x(0)$, being 
sampled randomly uniformly from the feasible set of the corresponding 
optimization problem. The average runtimes together with the 
observed 2.5\%- and 95.5\%-quantiles are shown in 
Figure \ref{fig:runtimes_vs_n}. We observe that solve times increase 
roughly linearly with $n$ for fixed $N$, $N_s$, and $q$.

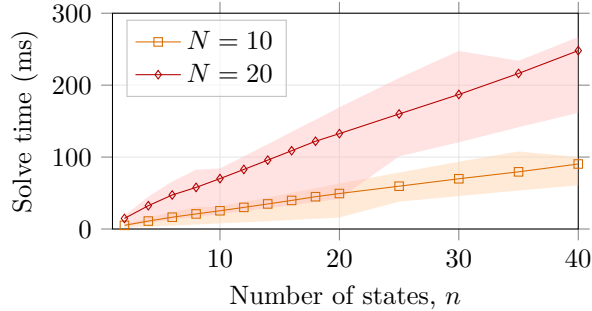
\begin{figure}
    \centering
    \tikzsetnextfilename{runtime-vs-state-dimension}
\begin{tikzpicture}
\begin{axis}[
    width=\PlotWidth,
    height=\MediumPlotHeight,
    xlabel={Number of states, $n$},
    ylabel={Solve time ($\unit{ms}$)},
    y label style={at={(-0.14,0.5)}},
    xmin=1, xmax=40,
    ymin=0, ymax=300,
    grid=both,
    major grid style={draw=gray!20},
    minor grid style={draw=gray!20},
    legend columns=1,
    legend style={
        draw=gray!50,
        fill=white,
        fill opacity=0.65,
        text opacity=1},
    legend pos=north west,
    tick align=outside,
    mark size=1.6pt,
]

\addplot+[
    name path=ci10upper,
    draw=none,
    mark=none,
    forget plot,
] coordinates {
(2,6.49949096)
(4,15.88834617)
(6,23.69362883)
(8,30.36852675)
(10,31.58789969)
(12,37.70733363)
(14,43.70298478)
(16,50.41458782)
(18,56.66126564)
(20,63.3638847)
(25,78.4814899)
(30,93.44876663)
(35,107.8185298)
(40,100.0509028)
% (45,123.5417566)
% (50,140.6198052)
};
\addplot+[
    name path=ci10lower,
    draw=none,
    mark=none,
    forget plot,
] coordinates {
(2,1.573397247)
(4,3.115175408)
(6,4.67782675)
(8,6.195818429)
(10,7.789244797)
(12,9.327317388)
(14,10.94757253)
(16,12.63607255)
(18,14.25488784)
(20,15.92298855)
(25,38.04331251)
(30,45.74253895)
(35,53.48425141)
(40,60.74433165)
% (45,62.40442845)
% (50,68.90167115)
};
\addplot+[
    draw=none,
    fill=orange!35,
    fill opacity=0.45,
    mark=none,
    forget plot,
] fill between[
    of=ci10upper and ci10lower,
];
\addplot+[
    color=orange!85!black,
    mark=square,
] coordinates {
(2,5.107149892)
(4,11.06344301)
(6,16.42574868)
(8,20.96268808)
(10,25.3577829)
(12,30.04975795)
(14,34.84084109)
(16,39.71724732)
(18,44.91171743)
(20,49.32903392)
(25,59.52909267)
(30,69.84594835)
(35,79.48826134)
(40,90.3741082)
% (45,96.13458241)
% (50,105.1865861)
};
\addlegendentry{$N=10$}

\addplot+[
    name path=ci20upper,
    draw=none,
    mark=none,
    forget plot,
] coordinates {
(2,18.44618125)
(4,44.91645037)
(6,66.1467229)
(8,82.37436321)
(10,84.18442363)
(12,100.9293455)
(14,117.9010198)
(16,135.1263854)
(18,151.9397955)
(20,169.1863281)
(25,209.9725153)
(30,247.4905729)
(35,233.8647309)
(40,266.8510062)
% (45,310.5768239)
% (50,377.2014667)
};
\addplot+[
    name path=ci20lower,
    draw=none,
    mark=none,
    forget plot,
] coordinates {
(2,4.404548618)
(4,8.544985737)
(6,12.78819967)
(8,16.41495449)
(10,20.56346291)
(12,24.75089961)
(14,28.93183679)
(16,33.07747916)
(18,37.47227707)
(20,41.88955898)
(25,101.0605906)
(30,120.5399806)
(35,141.4815865)
(40,161.5524333)
% (45,180.2981166)
% (50,185.882299)
};
\addplot+[
    draw=none,
    fill=red!25,
    fill opacity=0.45,
    mark=none,
    forget plot,
] fill between[
    of=ci20upper and ci20lower,
];
\addplot+[
    color=red!70!black,
    mark=diamond,
] coordinates {
(2,14.70170979)
(4,32.39951251)
(6,47.29163471)
(8,57.82333925)
(10,70.04572514)
(12,82.80066617)
(14,95.75295571)
(16,108.8454656)
(18,121.99544)
(20,132.6044243)
(25,159.9258411)
(30,187.0251058)
(35,216.29218)
(40,247.9923693)
% (45,276.1010639)
% (50,295.6914859)
};
\addlegendentry{$N=20$}

\end{axis}
\end{tikzpicture}
    \caption{Solve times against number of states. 
    (Solid lines) average runtimes over 1000 random feasible initial 
    conditions. (Shaded areas) 2.5\%- and 95.5\%-quantiles of the 
    recorded solve times.}
    \label{fig:runtimes_vs_n}
\end{figure}

\subsection{Sampled-data fractional optimal control}
The set of feasible initial states for the approximate optimal 
control problem with $N_s$ subintervals is
$\F_{N_s}^{x_0}
\dfn
\{x_0 : \exists u_0,\ldots,u_{N-1} \text{ with } (x_0,u_0,\ldots,u_{N-1})\in \F_{N_s}\},$
which is the projection of $\F_{N_s}$ onto $x_0$.
Fig.~\ref{fig:phase_profile_trajectories_feasibility} shows the 
feasible sets, $\F_{N_s}^{x_0}$, with $N_s=16$ and $N_s = 2$,
for a horizon of $N=20$.
The figure also shows two optimal state trajectories initialized 
at extreme points of $\F_{16}^{x_0}$ and demonstrates 
that the state constraints are satisfied in continuous 
time. 
The set $\F_{16}^{x_0}$ is slightly larger than $\F_{2}^{x_0}$, 
illustrating the reduction of conservatism as the number 
of subintervals increases. 
The continuous-time trajectories remain inside the state constraint set, 
$\mathcal{X}$, over the prediction horizon.
The optimal control signal is shown in 
Fig. \ref{fig:phase_profile_trajectories_feasibility:inputs}, which satisfies the input bounds.

\begin{figure}[ht]
    \centering
    \input{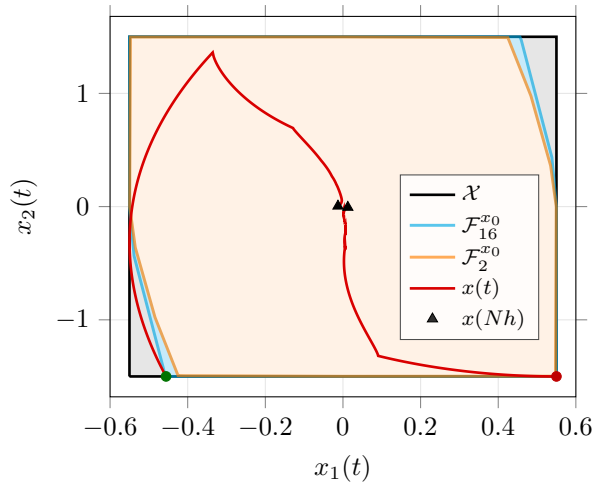}
    \caption{Two optimal state trajectories over a horizon $N=20$, feasible sets $\F^{x_0}_{2} \subseteq \F^{x_0}_{16}$, and the set of state constraints $\mathcal{X}$.}
    \label{fig:phase_profile_trajectories_feasibility}
\end{figure}

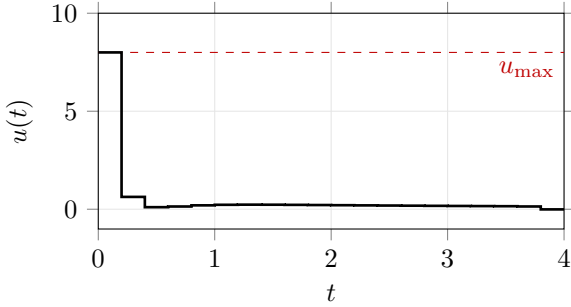
\begin{figure}[ht]
    \centering
    \tikzsetnextfilename{extreme-vertex-inputs}
\begin{tikzpicture}
\begin{axis}[
    width=\PlotWidth,
    height=\MediumPlotHeight,
    xmin=0,
    xmax=4,
    ymax=10,
    xlabel={$t$},
    ylabel={$u(t)$},
    tick align=outside,
    grid=both,
    major grid style={draw=gray!20},
    minor grid style={draw=gray!20},
    legend pos=north east,
    legend style={draw=black!70, fill=white, fill opacity=0.82, text opacity=1, font=\footnotesize},
    legend cell align=left,
]

% MAXIMUM u
\addplot [red!75!black, dashed, forget plot] coordinates {(0,8) (4,8)};
\node[anchor=north east, red!75!black] at (axis cs: 4, 8) {$u_{\rm max}$};

% \addplot [black, const plot, line width=1pt] coordinates {
%     (0,0.4960544796) 
%     (0.2,2.947374342) 
%     (0.4,1.292711291) 
%     (0.6,0.7976798111) 
%     (0.8,0.5942363084) 
%     (1,0.4828452512) 
%     (1.2,0.4098067316) 
%     (1.4,0.3571980132) 
%     (1.6,0.3172863642) 
%     (1.8,0.2859610959) 
%     (2,0.2607431693) 
%     (2.2,0.2400290773) 
%     (2.4,0.2227352215) 
%     (2.6,0.2081015056) 
%     (2.8,0.1955615688) 
%     (3,0.1846226342) 
%     (3.2,0.1747316771) 
%     (3.4,0.1655501415) 
%     (3.6,0.1627543768) 
%     (3.8,0.3210065394) 
%     (4,0.3210065394) 
% };
% \addlegendentry{Trajectory 1}

% -- This is like step() in MATLAB
\addplot [black, const plot, line width=1pt] coordinates {
    (0,7.999999995) 
    (0.2,0.6284966202) 
    (0.4,0.1035310213) 
    (0.6,0.1382997839) 
    (0.8,0.1977133291) 
    (1,0.2257055841) 
    (1.2,0.2327086773) 
    (1.4,0.2296597528) 
    (1.6,0.2224058831) 
    (1.8,0.2136651653) 
    (2,0.204644967) 
    (2.2,0.1958641617) 
    (2.4,0.1875295344) 
    (2.6,0.1797084015) 
    (2.8,0.1724177033) 
    (3,0.1656853526) 
    (3.2,0.1595531107) 
    (3.4,0.1534305044) 
    (3.6,0.1396993834) 
    (3.8,-0.009081673459) 
    (4,-0.009081673459) 
};
% \addlegendentry{Trajectory 2}

% \addplot [black!60, dashed, forget plot] coordinates {(0,-8) (4,-8)};

\end{axis}
\end{tikzpicture}
    \caption{Piecewise constant optimal control actions 
    corresponding to Fig. 
    \ref{fig:phase_profile_trajectories_feasibility} with 
    $x(0)$ being at the lower-left extreme point of 
    $\F_{16}^{x_0}$.}
    \label{fig:phase_profile_trajectories_feasibility:inputs}
\end{figure}

Furthermore, for the initial state 
$x(0)=(-0.55,-1.125)$, 
enforcing the state constraints only at the sampling instants 
yields a trajectory that violates the continuous-time constraints 
between samples. 
In particular, at $t^*\approx 0.065$ we obtain 
$x(t^*) {}\approx{} (-0.628,-0.241) \notin \mathcal{X}$. 
Since $x(0)\notin \F^{x_0}_{16}$, this point is 
infeasible for the tightened problem.
Likewise, starting from $x(0)=(-0.45947,-1.46975)$, 
which is an extreme point of $\F_{16}^{x_0}$ and applying
state constraints only at the discrete sampling time instants,
a violation is observed at $t^*\approx 0.059$,
where $x(t^*) {}\approx{} (-0.55486,-0.31061)\notin \mathcal{X}$.
Following \cite{SopasakisSarimveis2017}, if discretize 
the system dynamics using the Gr{\"u}nwald-Letnikov 
derivative and impose the constraints at discrete
times only, we obtain an optimal sequence of control actions
which leads to violations in the inter-sample intervals;
in particular, at $t^*\approx 0.06$, we have
$x(t^*)\approx(-0.55540, -0.30958) \notin \mathcal{X}$.

\section{Conclusions}
We proposed a tractable formulation of sampled-data optimal control for
linear fractional-order systems under continuous-time state constraints.
A key property was that $x(t)$ can be represented in a way---cf. 
\eqref{eq:key-system-representation}---
where the smooth and nonsmooth parts of the fractional response 
are separated.
Subsequently, we used Bernstein polynomials to construct 
polytopic approximations of the continuous-time state trajectories.
We showed that the conservatism of the proposed relaxation 
can be controlled by subdividing each sampling interval.
We further showed that under a Slater-type condition, as the number of 
subintervals goes to infinity, the approximate feasible sets converge 
to the exact semi-infinite feasible set in the Painlev{\'e}-Kuratowski sense.

The numerical examples support the theoretical findings. We have shown
that subdivision is the dominant mechanism for reducing
conservatism, while runtime scales gracefully with $N_s$, $N$, and $n$, 
and the tightened constraints enforce continuous-time feasibility.

Future work will focus on extensions to nonlinear 
fractional\hyp{}order systems, and on the design of model 
predictive controllers for sampled-data fractional-order 
linear and nonlinear systems.

\bibliographystyle{unsrtnat}
\bibliography{bibliography}
\end{document}